\documentclass[oneside,12pt]{amsart}
\usepackage{amssymb, mathrsfs, amscd}
\usepackage[all]{xy}

\usepackage{graphics}
\usepackage{graphicx}
\usepackage{enumerate}

\newtheorem{theorem}{Theorem}[section]  
\newtheorem{definition}[theorem]{Definition}
\newtheorem{lemma}[theorem]{Lemma}

\newtheorem{proposition}[theorem]{Proposition}
\newtheorem{corollary}[theorem]{Corollary}
\newtheorem*{corollary*}{Corollary}

\newtheorem{remark}[theorem]{Remark}

\newtheorem{question}[theorem]{Question}

\author{Dieter Degrijse}
\address{Department of Mathematics, Univeristy of Copenhagen, Denmark}%
\email{D.Degrijse@math.ku.dk}%

\thanks{The first author was supported by the Danish National Research Foundation through the Centre for Symmetry and Deformation (DNRF92) and the second by  Gobierno de Arag\'on, European Regional Development Funds and Ministerio de Educaci\'on, Cultura y Deporte through
MTM2010-19938-C03-03. }
\author{Conchita Mart{\'\i}nez-P{\'e}rez}
\address{Department of Mathematics, Univeristy of Zarargoza, Spain}
\email{conmar@unizar.es}
\title[]{Dimension invariants for groups admitting a cocompact model for proper actions}
\date{\today}

\newcommand{\mF}{\mathcal {F}}
\newcommand{\mFK}{\mathcal{F}_K}
\newcommand{\Z}{\mathbb Z}

\newcommand{\orb}{\mathcal{O}_{\mF}G}

\newcommand{\subK}{\mathcal{S}_{\mFK}G}
\newcommand{\orbmod}{\mbox{Mod-}\mathcal{O}_{\mF}G}
\newcommand{\zmod}{\mathbb{Z}\mbox{-Mod}}

\bigskip
\begin{document}
\maketitle
\begin{abstract}
Let $G$ be a group that admits a cocompact classifying space for proper actions $X$. We derive a formula for the Bredon cohomological dimension for proper actions of $G$ in terms of the relative cohomology with compact support of certain pairs of subcomplexes of $X$. We use this formula to compute the Bredon cohomological dimension for proper actions of fundamental groups of non-positively curved simple complexes of finite groups. As an application we show that if a virtually torsion-free group acts properly and chamber transitively on a building, its virtual cohomological dimension coincides with its Bredon cohomological dimension. This covers the case of Coxeter groups and graph products of finite groups.
\end{abstract}
\section{Introduction}
Let $G$ be a discrete infinite group. By a proper $G$-CW-complex we mean a $G$-CW-complex with finite point stabilizers. One says a $G$-CW complex is cocompact if the orbit space $G \setminus X$ is compact. A classifying space for proper actions for $G$ is a proper $G$-CW-complex $X$ such that the fixed point subcomplexes $X^H$ are contractible for every finite subgroup $H$ of $G$. Such a space $X$ is also called a model for $\underline{E}G$. These spaces appear naturally throughout geometric group theory and related areas, and their study is of fundamental importance for our understanding of infinite discrete groups and their actions on spaces. We refer the reader to the survey paper \cite{Luck2} for more information. An interesting invariant of the group $G$ is its geometric dimension for proper actions. This invariant, denoted by $\underline{\mathrm{gd}}(G)$, is by definition the smallest possible dimension that a model for $\underline{E}G$ can have. The invariant $\underline{\mathrm{gd}}(G)$ is often studied via its algebraic counterpart, which is the Bredon cohomological dimension $\underline{\mathrm{cd}}(G)$ of the group $G$ (see Section 2). These two notions of dimension coincide (see \cite{LuckMeintrup}), except for the possibility that one could have $\underline{\mathrm{cd}}(G)=2$ but $\underline{\mathrm{gd}}(G)=3$ (see \cite{BradyLearyNucinkis}).\\

It is a well known fact that if $G$ admits a cocompact contractible proper $G$-CW-complex $X$, then the cohomology of $G$ with group ring coefficients can be computed as the cohomology with compact support of $X$. If $G$ is virtually torsion-free, this implies that the virtual cohomological dimension of $G$ can be computed as 
   \begin{equation} \label{eq: vcd} \mathrm{vcd}(G)=\max \{  n \in \mathbb{N} \ | \  \mathrm{H}^n_c(X)\neq 0  \} . \end{equation}

In this paper we derive a formula that enables one to compute the Bredon cohomological dimension of a group $G$ via the relative cohomology with compact support of certain pairs of subcomplexes of a cocompact model for $\underline{E}G$, assuming this exists.
\begin{theorem}\label{th: bredon}{\rm[Cor.~\ref{cor: key cor}] }Let $G$ be a group that admits a cocompact model $X$ for $\underline{E}G$. Then

\[    \underline{\mathrm{cd}}(G)= \max\{n \in \mathbb{N} \ | \ \exists K \in \mathcal{F} \ \mbox{s.t.} \ \mathrm{H}_c^{n}(X^K,X^K_\mathrm{sing})  \neq 0\}   \] 
where  $\mathcal{F}$ is the family of finite subgroups of $G$ and $X^K_\mathrm{sing}$ is the subcomplex of $X^{K}$ consisting of those cells that are fixed by a finite subgroup of $G$ that strictly contains $K$.
\end{theorem}
Using Theorem \ref{th: bredon}, we are able to compute the Bredon cohomological dimension of certain simple complexes of groups over a finite poset $\mathcal{Q}$ (see Definition \ref{def: complex}).
\begin{theorem}\label{th: nonpos}{\rm[Th.~\ref{th: main app nonpos}]} Let $G(\mathcal{Q})$ be a non-positively curved simple complex of finite groups with $K=|\mathcal{Q}|$ simply connected and fundamental group $G$. Then 

\[    \underline{\mathrm{cd}}(G)=\max\{ n \in \mathbb{N} \ | \ {\mathrm{H}}^{n}(K_{\Omega_J},K_{>\Omega_J})\neq 0 \ \mbox{for some} \ J \in \mathcal{S}    \} \]
where $\Omega_J$ is the set of elements $U$ of $\mathcal{Q}$ for which the local group $P_U$ equals $P_J$ as subgroups of $G$, and $K_{\Omega_J}$ $(K_{> \Omega_J})$ is the geometric realization of the subposet of $\mathcal{Q}$ consisting of elements that are (strictly) larger than some element of $\Omega_J$.

\end{theorem}
As illustrated in the following corollary, this theorem can for example be used to determine when the fundamental group of a non-positively curved simple complex of finite groups is not virtually free.
\begin{corollary} {\rm[Cor.~\ref{ngon}]} The geometric dimension for proper actions of the fundamental group $G$ of a non-positively curved $n$-gon of finite groups is $2$. In particular, $G$ is not virtually free. 

\end{corollary}
In general, the invariant $\underline{\mathrm{cd}}(G)$ is very hard to determine and one is therefore also interested in simpler algebraic invariants that approximate $\underline{\mathrm{cd}}(G)$. The relative cohomological dimension $\mathcal{F}\mathrm{cd}(G)$ (see Definition \ref{def: rel}) (or the virtual cohomological dimension $\mathrm{vcd}(G)$ which coincides with $\mathcal{F}\mathrm{cd}(G)$ if $G$ is virtually torsion-free) can fulfill this role since it always satisfies $\mathcal{F}\mathrm{cd}(G)\leq \underline{\mathrm{cd}}(G)$.  In general, this is not a very good approximation. Indeed, in the celebrated paper \cite{LearyNucinkis} Leary and Nucinkis constructed examples of groups $G$ for which $\underline{\mathrm{cd}}(G)$ is finite but strictly larger than $\mathcal{F}\mathrm{cd}(G)$. Moreover, they show that the gap between the two can be arbitrarily large. However, in these examples the groups under consideration do not admit a cocompact model for $\underline{E}G$. In fact, it is still an open problem whether or not the equality $\mathcal{F}\mathrm{cd}(G)= \underline{\mathrm{cd}}(G)$ holds for all groups admitting a cocompact model for $\underline{E}G$. This is also not known for specific classes of group such as word-hyperbolic groups and CAT(0)-groups. It is known to be true for other important families of groups such as 
 elementary amenable groups of type $\operatorname{FP}_\infty$ (see \cite[Theorem 1.1]{KMPN} and the remark before Conjecture 1.4), $\mathrm{SL}(n,\mathbb{Z})$ (e.g. see \cite[Remark 5.25 ]{Luck2}), mapping class groups (\cite{Armart}) and outer automorphism groups of free groups (e.g. see \cite[4.9]{Luck2},\cite{Vogtmann}). Using a result by St.~John-Green, we point out that the $\mathcal{F}\mathrm{cd}(G)$ can be computed via the compactly supported cohomology of a contractible proper $G$-CW-complex (see Theorem \ref{fcd}). Combining this with Theorem \ref{th: nonpos}, we show that the equality $\mathcal{F}\mathrm{cd}(G)= \underline{\mathrm{cd}}(G)$ also holds for groups acting properly and chamber-transitively on a building of type $(W,S)$, where $(W,S)$ is  a finitely generated Coxeter system.

\begin{corollary}{\rm[Th.~\ref{building}]} Let $G$ be a discrete group that acts properly and chamber-transitively on a building of type $(W,S)$. Then 
\[\mathrm{vcd}(W)=\mathcal{F}\mathrm{cd}(G)=\underline{\mathrm{cd}}(G)= \max\{ n \in \mathbb{N} \ | \ \overline{\mathrm{H}}^{n-1}(K_{>J})\neq 0  \ \mbox{for some} \ J \in \mathcal{Q}  \}\]
where $\mathcal{Q}$ is the poset of spherical subsets of $S$ and $K=|\mathcal{Q}|$.
\end{corollary}
\noindent This corollary shows that for finitely generated Coxeter groups and finite graph products of finite groups, the Bredon cohomological dimension and the virtual cohomological dimension coincide.

\section{Bredon cohomological dimension}
An important algebraic tool to study classifying spaces for proper actions is Bredon cohomology. This cohomology theory was introduced by Bredon in \cite{Bredon} for finite groups 
as a mean to develop an obstruction theory for equivariant extension of maps. It was later generalized to arbitrary groups by L\"{u}ck with applications to finiteness conditions  (see \cite[section 9]{Luck},  \cite{Luck1} and \cite{LuckMeintrup}).
We will recall some basic notions of this theory, mainly to establish notation.  For more details, we refer the reader to \cite{FluchThesis} and \cite{Luck}. \\

Let $G$ be a discrete group and let $ \mathcal{F}$ be the family of finite subgroups of $G$. The trivial subgroup of $G$ will be denoted by $1$. The orbit category $\orb$ is the category with $G$-sets $G/H$, $H\in\mathcal{F}$ as objects and $G$-maps as morphisms. Note that the set of morphisms $\mathrm{Mor}(G/H,G/K)$ can be identified with the fixed point set $(G/K)^H$ and that a $G$-map $G/H \rightarrow G/K : H \mapsto xK$ will be denoted by either $G/H \xrightarrow{x} G/K$ or $xK \in (G/K)^H$. A right $\orb$-module is a contravariant functor \[M: \orb \rightarrow \zmod.\] The right $\orb$-modules are the objects of an abelian category $\orbmod$, whose morphisms are natural transformations. The set of morphisms  between two objects $M,N\in\orbmod$ is denoted by $\mathrm{Hom}_{\orb}(M,N)$. This category has free objects; more precisely, its free objects are those of the form
$$P_K(-)=\Z[-,G/K]$$
where $K\in\mathcal{F}$ and $\Z[G/H,G/K]$ is the free $\Z$-module with basis the set of $G$-maps $G/H\to G/K$. A sequence of modules in $\orbmod$ is said to be exact if it is exact when evaluated at every object. There is also a Yoneda lemma that allows one to construct free (projective) resolutions in a similar way as in the case of ordinary cohomology. The Bredon cohomological dimension is the smallest possible length of a free (projective) resolution of the trivial object in $\orbmod$, which is the constant functor $\underline{\Z}:G/H\to\Z$. The $n$-th Bredon cohomology of $G$ with coefficients in a right $\orb$-module $M$ is denoted by
$$\mathrm{H}^n_{\mathcal{F}}(G,M)$$
and is defined as the $n$-th cohomology of the cochain complex obtained applying the contravariant functor $\text{Hom}_{\orb}(-,M)$ to a free (projective) resolution of $\underline{\Z}$.

We now continue assuming that the group $G$ admits a cocompact classifying space for proper actions $X$. In this case, the free $\mathcal{O}_{\mathcal{F}}G$-resolution $C_{\ast}(X^{-}) \rightarrow \underline{\mathbb{Z}}$ is finite dimensional and consists of finitely generated free modules. The existence of such a free resolution implies that the Bredon cohomology of $G$ commutes with  arbitrary direct sums of coefficient modules. And by a standard argument we get
\begin{equation} \label{eq: old cd from}\underline{\mathrm{cd}}(G) = \sup\{ n \in \mathbb{N} \ | \ \mathrm{H}^n_{\mathcal{F}}(G,P_K)\neq 0 \ \mbox{for some} \ K \in \mathcal{F} \}. \end{equation}
Our goal for the remainder of this section will be to derive a formula for $\underline{\mathrm{cd}}(G) $ involving the compactly supported cohomology of certain pairs of subcomplexes of $X$.
\begin{definition} \rm Given $K \in \mathcal{F}$, we define $N_K$ to be the submodule of $P_K$ generated by all $H \in \mathcal{F}$ that are conjugate to a proper subgroup of $K$. We define $Q_K$ as the cokernel of the natural inclusion $N_K \rightarrow P_K$. Hence, there is a short exact sequence of $\orb$-modules
\begin{equation} \label{eq: short exact seq free}   0 \rightarrow N_K \rightarrow P_K \rightarrow Q_K \rightarrow 0.   \end{equation}
Note that
\[Q_K(G/L)=\Bigg\{
\begin{aligned}
&\mathbb{Z}[G/L,G/K] &\text{ if } L=_GK\\
&0&\text{ otherwise}\\
\end{aligned}\] 
and $P_1=Q_{1}$. Here, $L=_GK$ means that $L$ and $K$ are conjugate in $G$.
\end{definition}
The next proposition shows that in (\ref{eq: old cd from}), we may replace $P_K$ by $Q_K$.
\begin{proposition} \label{prop: mod to dim} If a group $G$ admits a cocompact model for $\underline{E}G$, then
\[\underline{\mathrm{cd}}(G) = \sup\{ n \in \mathbb{N} \ | \ \mathrm{H}^n_{\mathcal{F}}(G,Q_K)\neq 0 \ \mbox{for some} \ K \in \mathcal{F} \}.\]
\end{proposition}
\begin{proof}
Since $G$ admits a finite dimensional model for $\underline{E}G$, we have $\underline{\mathrm{cd}}(G)=n<\infty$. We need to show that there exists a $K\in\mathcal{F}$ such that
$\mathrm{H}^n_\mathcal{F}(G,Q_K)\neq 0$. Choose  $K \in \mathcal{F}$ with $|K|$ minimal, under all $K \in \mathcal{F}$ for which $\mathrm{H}^n_{\mathcal{F}}(G,P_K)\neq 0$.  We claim that one of the following two statements must be true.
\begin{itemize}
\item[(a)] $K=1$;
\item[(b)] for any $M\in \orbmod$ such that $M(G/L)=0$ for every $L\in\mathcal{F}$ with $|L|\geq|K|$, one has $\mathrm{H}^n_\mathcal{F}(G,M)= 0$.
\end{itemize}
To prove the claim, we may assume $K\neq 1$. If there is a module $M$ as in (b) but with $\mathrm{H}^n_\mathcal{F}(G,M)\neq 0$, we may form a free cover $P$ of $M$ consisting of a sum of free modules based at subgroups $L$ with $|L|<|K|$. Since the Bredon cohomology of $G$ commutes with direct sums and $\underline{\mathrm{cd}}(G)=n$, there must be some $L \in \mathcal{F}$ with $|L|<|K|$ such that $\mathrm{H}^n_\mathcal{F}(G,P_L)\neq 0$. This however contradicts the minimality of $K$, and proves that (b) holds. Hence, the claim is proven. In the case when $K=1$ we have $Q_K=P_K$.  If $K\neq 1$, the long exact sequence of cohomology functors obtained from (\ref{eq: short exact seq free}) implies that either $\mathrm{H}^n_\mathcal{F}(G,Q_K)\neq 0$ or $\mathrm{H}^n_\mathcal{F}(G,N_K)\neq 0$. But this last possibility can not happen since $N_K(L)=0$  for every $L\in\mathcal{F}$ with $|L|\geq|K|$. 
\end{proof}
\begin{definition}\rm For a proper $G$-CW-complex $X$ and a finite subgroup $K$ of $G$, we define
\[    X^K_{\mathrm{sing}}=\bigcup_{K\subsetneq L \in \mathcal{F}} X^L \subseteq X^K.  \]

\end{definition}
Given a CW-complex $X$, let $C_{\ast}(X)$ be its cellular chain complex and let $\mathrm{Hom}_c(C_n(X),\mathbb{Z})$ be the subspace of $\mathrm{Hom}(C_n(X),\mathbb{Z})$
consisting of those cochains that vanish on all but finitely many $n$-cells. The cohomology of $\mathrm{Hom}_c(C_{\ast}(X),\mathbb{Z})$ is the cohomology with compact support of $X$, denoted by $\mathrm{H}^{\ast}_c(X)$. If $A$ is a CW-subcomplex of $X$, we obtain a short exact sequence of chain complexes
\[      0\rightarrow C_c^{\ast}(X,A) \rightarrow \mathrm{Hom}_c(C_{\ast}(X),\mathbb{Z}) \rightarrow \mathrm{Hom}_c(C_{\ast}(A),\mathbb{Z}) \rightarrow 0 .   \]
The relative cohomology with compact support of the pair $(X,A)$ is by definition the cohomology of $C^{\ast}_c(X,A)$. \\

We can now state the main result of this section.
\begin{theorem}\label{th: compact support bredon} Let $G$ be a group that admits a cocompact model $X$ for $\underline{E}G$. Then for every $K\in\mathcal{F}$, we have
\[  \mathrm{H}^{\ast}_{\mathcal{F}}(G,Q_K)\cong  \mathrm{H}_c^{\ast}(X^K,X^K_\mathrm{sing})  \]
where $\mathrm{H}_c^\ast(X^K,X^K_\mathrm{sing})$ is the relative cohomology with compact support of the pair $(X^K,X^K_\mathrm{sing})$.
\end{theorem}
The following corollary will be a useful tool for determining the Bredon cohomological dimension of certain groups.
\begin{corollary}\label{cor: key cor} Let $G$ be a group that admits a cocompact model $X$ for $\underline{E}G$. Then,
\[    \underline{\mathrm{cd}}(G)= \max\{n \in \mathbb{N} \ |  \ \mathrm{H}_c^{n}(X^K,X^K_\mathrm{sing})  \neq 0 \ \mbox{for some} \ K \in \mathcal{F}\}.   \] 

\end{corollary}
\begin{proof} This is immediate from Proposition \ref{prop: mod to dim} and Theorem \ref{th: compact support bredon}. 

\end{proof}
The rest of the section is devoted to proving Theorem \ref{th: compact support bredon}.\\
\begin{definition} \rm Let $K \in \mathcal{F}$ and consider the collection of subgroups
\[    \mFK=\{   H \in \mathcal{F}  \ | \ K \subseteq H \}.  \]  The subgroup category $\subK$ is the category defined as follows. The objects are $G/H$, for every $H \in \mFK$, and there is a unique morphism $G/H \rightarrow G/L$ if and only if $H\subseteq L$.
\end{definition}
A right $\subK$-module is a contravariant functor from $\subK$ to the category of $\Z$-modules. They form a category 
denoted by $\mbox{Mod-}\subK$. We may do homological algebra in this category in an entirely similar way as we did in $\mbox{Mod-}\orb$. The free right $\subK$-module based at $G/L$ will be denoted by $S_L$, it is given by $S_L(G/H)=\Z$ if $H\subseteq L$, $S_L(G/H)=0$ otherwise. Moreover, there is an obvious inclusion of categories
\[   \subK  \rightarrow \orb    \] 
that gives rise to a (co)induction and a restriction functor on the level of module categories. The restriction functor will be omitted from the notation.

\begin{definition} \rm Let $M$ be a right $\orb$-module and let $N$ be a right $\subK$-module. We define \[ \mathcal{H}\mathrm{om}_{\subK}(M,N)\] to be the submodule of $ \mathrm{Hom}_{\subK}(M,N) $   containing all functors $F$ satisfying the following property: for all $H,L \in \mathcal{F}_K$ and all $m \in M(G/L)$, $F(G/H)(M(gL)(m))\neq 0$ for at most finitely many $G/H\buildrel{g}\over\to G/L$.

\end{definition}
The definition of $ \mathcal{H}\mathrm{om}_{\subK}(M,N)$, which may seem artificial to the reader, is designed as a generalization to the Bredon setting of the group $\mathrm{Hom}_c(M,\Z)$ from \cite[Ch. VIII Lemma 7.4.]{Brown}. Indeed, we prove the following.

\begin{lemma} \label{lemma: nat transf}For each $M \in \orbmod$ and every $K \in \mathcal{F}$, there is a natural isomorphism
\[     \mathrm{Hom}_{\orb}(M,Q_K) \cong  \mathcal{H}\mathrm{om}_{\subK}(M,S_K) .  \]

\end{lemma}
\begin{proof} Note that $S_K(G/L)$ vanishes  unless $L= K$. Consider the natural transformation of $\subK$-modules $$\delta_K : Q_K \rightarrow S_K $$ defined by 
\[\delta_K(G/K)(gK)=\Bigg\{
\begin{aligned}
&1 \text{ if } g \in K \\
&0\text{ otherwise}\\
\end{aligned}\]
and $\delta_K(G/H)=0$ if $H \neq K$.  Now define the natural map
\[  \Psi: \mathrm{Hom}_{\orb}(M,Q_K) \rightarrow   \mathrm{Hom}_{\mathcal{S}_{\mathcal{F}_K}G}(M,S_K) : F \mapsto F_0=\delta_K \circ F.\]
We will first show that $\Psi$ is an inclusion. Take $0 \neq F \in  \mathrm{Hom}_{\orb}(M,Q_K) $.
Since $F(G/H)$ is zero for all $H\neq_{G} K$, we must have $F(G/K)(m)\neq 0$ for some $m \in M(G/K)$. Let \[F(G/K)(m)=\sum_{gK \in (G/K)^{K}}n_ggK \in \mathbb{Z}[(G/K)^K]\] with  $n_{g_0}\neq 0$. 
Then $F_0(G/K)(M(g_0^{-1}K)(m))=n_{g_0}$, so $F_0\neq 0$ proving that $\Psi$ is injective. Next we show that the image of $\Psi$ is contained in $ \mathcal{H}\mathrm{om}_{\mathcal{S}_{\mathcal{F}_K}G}(M,S_K)$. Take $F \in  \mathrm{Hom}_{\orb}(M,Q_K)$ and consider $\Psi(F)=F_0$. If $F_0(G/H)$ is non-zero, this implies that $H=K$. So we just need to show that for each $m \in M(G/K)$, $F_0(G/K)(M(gK)(m))$ is non-zero for at most finitely many $gK \in (G/K)^K$. Since 
\[   F(G/K)(m)= \sum_{gK \in (G/K)^K} F_0(G/K)(M(g^{-1}K)(m))gK,    \]
this is clear. This formula also proves that $\Psi$ surjects onto the space $\mathcal{H}\mathrm{om}_{\mathcal{S}_{\mathcal{F}_K}G}(M,S_K) $. Indeed, one easily verifies that given a functor $F_0 \in \mathcal{H}\mathrm{om}_{\mathcal{S}_{\mathcal{F}_K}G}(M,S_K) $, this formula defines an $F \in \mathrm{Hom}_{\orb}(M,Q_K)$ such that $\Psi(F)=F_0$, by construction.
\end{proof}
\begin{definition} \rm Let $N$ be a right $\subK$-module and let $T \in \mathcal{F}$. We define \[\mathrm{Hom}^c_{\subK}(P_T,N) \]
to be the submodule of $ \mathrm{Hom}_{\subK}(P_T,N)$ containing all functors $F$ such that for all $H \in \mathcal{F}_K$,  $F(G/H)(gT)\neq 0$ for at most finitely many $gT \in (G/T)^H$.

\end{definition}
The next lemma shows that the submodules $\mathrm{Hom}^c_{\subK}(P_T,N) $ and  $ \mathcal{H}\mathrm{om}_{\subK}(P_T,N)$ of $ \mathrm{Hom}_{\subK}(P_T,N)$ coincide.
\begin{lemma} \label{lemma: coincide} Let $N$ be a right $\subK$-module and let $T \in \mathcal{F}$. Then
\[    \mathcal{H}\mathrm{om}_{\subK}(P_T,N)=  \mathrm{Hom}^c_{\subK}(P_T,N).   \]

\end{lemma}
\begin{proof}
We may assume that ${}^{y}T=yTy^{-1} \in \mFK$ for some $y \in G$. Otherwise  both  $ \mathrm{Hom}^c_{\subK}(P_T,N)$ and $\mathcal{H}\mathrm{om}_{\subK}(P_T,N)$ are zero, and there is nothing to prove. Let $H \in \mFK$. Since $gT=P_T(G/H \xrightarrow{gy^{-1}} G/{}^{y}T)(yT)$ for every $gT \in (G/T)^H$, it follows that  \[\mathcal{H}\mathrm{om}_{\subK}(P_T,N)\subseteq  \mathrm{Hom}^c_{\subK}(P_T,N).\]
Now let $F \in \mathrm{Hom}^c_{\subK}(P_T,N)$, choose $H,L \in \mathcal{F}_K$,  \[\sum_{gT \in (G/T)^L}n_g gT \in P_T(G/L)\]
and suppose $F(G/H)$ is zero everywhere expect possible on $g_1T,\ldots,g_kT \in (G/T)^H$. Since $T$ is a finite group, at most a finite number of different $xL \in (G/L)^H$ can be chosen such that \[\sum_{gT \in (G/T)^L}n_g xgT \] has support non-disjoint form $\{g_1T,\ldots,g_kT\} $, proving that 
\[    \mathrm{Hom}^c_{\subK}(P_T,N)\subseteq \mathcal{H}\mathrm{om}_{\subK}(P_T,N).  \]
\end{proof}

\begin{proof}[proof of Theorem \ref{th: compact support bredon} ]
Let $C_{\ast}(X^{-})\rightarrow \mathbb{Z}$ be the free $\orb$-resolution obtained from $X$. By the assumptions on $X$, this resolution is finite dimensional and each module $C_n(X^-)$ is a finite sum of free $\orb$-modules of the form $P_H(G/-)$. Moreover, we have $\mathrm{H}^{m}_{\mathcal{F}}(G,Q_K)=\mathrm{H}^{m}(\mathrm{Hom}_{\orb}(C_{\ast}(X^{-}),Q_K))$ for every $m \in \mathbb{N}$, so it follows from Lemma \ref{lemma: nat transf} that
\[  \mathrm{H}^{m}_{\mathcal{F}}(G,Q_K)=\mathrm{H}^{m}(\mathcal{H}\mathrm{om}_{\subK}(C_{\ast}(X^{-}),S_K))     \]
for every $m \in \mathbb{N}$.

Note that $S_K$ vanishes everywhere except at $K$, where it equals $\mathbb{Z}$. Hence, by evaluating an $F \in \mathcal{H}\mathrm{om}_{\subK}(C_{\ast}(X^{-}),S_K) $ at $K$, we obtain a natural injection
\[ \alpha: \mathcal{H}\mathrm{om}_{\subK}(C_{\ast}(X^{-}),S_K)\rightarrow \mathrm{Hom}(C_{\ast}(X^K),\mathbb{Z}): F \mapsto F(G/K). \]
Since each $C_n(X)$ is a finitely generated free module, it follows from Lemma \ref{lemma: coincide} that the image of $\alpha$ is contained in $\mathrm{Hom}_c(C_{\ast}(X^K),\mathbb{Z})$. Moreover,
since $S_K(L)=0$ for all $K \subsetneq L$, it is now an easy matter to verify that the image of $\alpha$ coincides with $C^{\ast}_c(X^K,X^{K}_{\mathrm{sing}})$. We conclude that there is an natural isomorphism
\[\mathcal{H}\mathrm{om}_{\subK}(C_{\ast}(X^{-}),S_K)\cong C^{\ast}_c(X^K,X^{K}_{\mathrm{sing}})\]
which implies that
\[  \mathrm{H}^{\ast}_{\mathcal{F}}(G,Q_K)\cong  \mathrm{H}_c^{\ast}(X^K,X^K_\mathrm{sing}).  \]
\end{proof}

\section{Simple complexes of finite groups}
We start this section by recalling the definition and some basic properties of simple complexes of finite groups. We use \cite[Ch. II.12]{BridHaef} as a basic reference, and refer the reader there for more details. \\

Throughout, let $\mathcal{Q}$ be a finite poset.  
\begin{definition} \rm \label{def: complex}By a simple complex of finite groups $G(\mathcal{Q})$, we mean a collection of finite groups $\{P_{J}\}_{J \in \mathcal{Q}}$ such that
\begin{itemize}
\item[(1)] for each $J < T$ in $\mathcal{Q}$, one has an injective non-surjective homomorphism \[\varphi_{J,T}: P_J \rightarrow P_T\]
\item[(2)] if $J< T <V $ in $\mathcal{Q}$, then $\varphi_{J,V}=\varphi_{T,V}\circ \varphi_{J,T}$.
\end{itemize}
The group $P_J$ is called the local group at $J \in \mathcal{Q}$.  Associated to $G(\mathcal{Q})$, one has the direct limit, or fundamental group,
\[   \widehat{G(\mathcal{Q})}= \lim_{J \in \mathcal{Q}} P_J  \]
which is obtained by taking the free product of the local groups $P_J$ for all $J \in \mathcal{Q}$ and then adding the relations 
\[    \varphi_{J,T}(g)=g \ \ \ \ \ \mbox{for all $J < T \in \mathcal{Q}$, $g \in P_J$}. \]
Note that there is a canonical homomorphism $i_J: P_J \rightarrow  \widehat{G(\mathcal{Q})}$ for every $J \in \mathcal{Q}$. One says the simple complex of groups $G(\mathcal{Q}) $ is strictly developable if all the maps $i_J$ are injective. 

From now on, assume that $G(\mathcal{Q}) $ is a strictly developable complex of finite groups and denote $G= \widehat{G(\mathcal{Q})}$. We will identify $P_J$ with its image under $i_J$, and omit the map $i_J$ from the notation. Consider  the poset $\mathcal{P}$  whose elements are the pairs $(gP_{J},J)$ for all $J \in \mathcal{Q}$ and all $gP_{J}$ in $G/P_J$, such that 
\begin{equation} \label{development}(gP_{J},J) < (g'P_{T},T) \ \ \ \Longleftrightarrow J <T \ \mbox{and} \ gP_T=g'P_T.    \end{equation}

There is a surjective map of posets $\pi: \mathcal{P}\rightarrow \mathcal{Q}: (gP_J,J) \mapsto J$ that has a spitting $s: \mathcal{Q} \rightarrow \mathcal{P}: J \mapsto (P_J,J) $. The development of $G(\mathcal{Q})$ is the geometric realization $X=|\mathcal{P}|$ of the poset $\mathcal{P}$. The geometric realization $|\mathcal{Q}|$ of $\mathcal{Q}$ will be denoted by $K$. Given a simplex $\sigma$ of $K$, we denote the smallest vertex of $\sigma$ by $S(\sigma)$. Using the geometric realization of the poset map $s$, we view $K$ as a subcomplex of $X$. More generally, if $T$ is a subcomplex of $K$ we denote by $gT$ the image of $T$ under the geometric realization of
$$s_g:\mathcal{Q}\to\mathcal{P}:J\mapsto(gP_J,J).$$

The group $G$ acts simplicially and admissibly on $X$ (i.e. a simplex is fixed if and only if it is fixed pointwise). The stabilizer of a simplex $(g_1P_{J_1},J_1)<\ldots < (g_nP_{J_n},J_n)$ is precisely the conjugate ${g_1}P_{J_1}g_1^{-1}$ of the local group $P_{J_1}$. Note that $K$ is a strict fundamental domain for the action of $G$ on $X$ and that the orbit space $G \setminus X$ is homeomorphic to $K$. In particular, we conclude that $X$ is a proper cocompact $G$-CW-complex.

\end{definition}
\begin{remark}\rm  In \cite[Ch. II.12]{BridHaef}, a simple complex of groups is defined using a different order convention and without the non-surjectivity assumption, i.e. to $J<T$ in $\mathcal{Q}$ they associate an injective homomorphism $\varphi_{J,T}: P_{T} \rightarrow P_J$. We use our reversely ordered definition because it is more convenient for our purposes. Of course, the two definitions are equivalent as one can just replace the poset $\mathcal{Q}$ by the poset $\mathcal{Q}^{\mathrm{op}}$ with the opposite ordering. The added assumption of non-surjectivity is just a minor technical assumption needed to avoid certain degenerate cases.
\end{remark}
For the remainder of this section, let $G(\mathcal{Q})$ be a strictly developable simple complex of finite groups with development $X$ and direct limit $G=\widehat{G(\mathcal{Q})}$. Our goal is to express the compactly supported cochain complex of the pairs $(X^{P_J},X_{\mathrm{sing}}^{P_J})$ in terms of cochain complexes of certain pairs of subcomplexes of $K$.  Our method is based upon a method used by Davis in \cite{Davis98} to compute the compactly supported cohomology of the Davis complex of a Coxeter group. \\

Let $J \in \mathcal{Q}$. Denote by $F(J)$ the fixed poset of the action of $P_J$ on $\mathcal{P}$, i.e.
 \begin{eqnarray*}
  F(J)&=&\{  (gP_{V},V) \in \mathcal{P} \ | \ P_JgP_V = gP_V  \}  \\
 &=&\{  (gP_{V},V) \in \mathcal{P}\ | \ P_J \subseteq  gP_Vg^{-1} \}. 
 \end{eqnarray*} 
 Note that $X^{P_J}=|F(J)|$ and that 
 $X^{P_J}_{\mathrm{sing}}$ is the geometric realization of the poset
\[ F(J)_{\mathrm{sing}}=\{  (gP_{V},V) \in \mathcal{P}\ | \ P_J \subsetneq gP_Vg^{-1} \} . \]
This means that any simplex in $X^{P_J}\setminus X^{P_J}_{\mathrm{sing}}$ has a smallest vertex of the form $(gP_{U},U)$ with $P_J = gP_Ug^{-1}$. Now fix a total ordering $\preceq$ on $G$ and define
$$\mathcal{L}(J)=\{g\in G\mid g^{-1}P_Jg\text{ is a local group and $g$ is the largest element of }P_Jg\}.$$ 
Observe that if $(gP_V,V)\in F(J)$ and $g^{-1}P_Jg$ is not local, then $(gP_V,V)\in F(J)_{\mathrm{sing}}$. 

For any subset $\Omega\subseteq\mathcal{Q}$ we denote by $K_{\Omega}$, resp. $K_{>\Omega}$, the geometric realizations of the following subposets

$$K_{\Omega}=|\{V\in\mathcal{Q}\mid U\leq V\text{ for some }U\in\Omega\}|,$$
$$K_{>\Omega}=|\{V\in\mathcal{Q}\mid U< V\text{ for some }U\in\Omega\}|.$$

For $g\in\mathcal{L}(J)$, we define
$$\Omega_g=\{U\in\mathcal{Q}\mid P_U=g^{-1}P_Jg\}.$$
We stress that if $U$ and $V$ are two distinct element in $\Omega_g$, then the groups $P_V$ and $P_U$ coincide as subgroups of $G$ but $(gP_U,U)$ and $(gP_V,V)$ are two distinct elements of the poset $F(J)$. Note also that $gK_{\Omega_g}\subseteq X^{P_J}\cap gK$, but this could be a proper inclusion due to the following. For any $(gP_V,V)\in F(J)$, we have $g^{-1}P_Jg\subseteq P_V$. Hence, for any $U\in\Omega_g$ we get $P_U\subseteq P_V$. But we cannot deduce from this that $U\leq V$. However, we do have the next  result.

\begin{lemma}\label{sets} For every $J \in \mathcal{Q}$ and every $g \in \mathcal{L}(J)$, we have

$$(X^{P_J}_{\mathrm{sing}}\cap gK)\cup(gK_{\Omega_g})=X^{P_J}\cap gK \ \ \mbox{and} \ \  (X^{P_J}_{\mathrm{sing}}\cap gK)\cap(gK_{\Omega_g})=gK_{>\Omega_g}.$$
\end{lemma}
\begin{proof} The first equality is a consequence of the obvious fact  that for any $(gP_U,U)\in F(J)-F(J)_{\mathrm{sing}}$, $U\in\Omega_g$. For the second equality, note first that if $U<V\in\mathcal{Q}$ with $U\in\Omega_g$, then $g^{-1}P_Jg=P_U\subsetneq P_V$ thus $(gP_V,V)\in F(J)_{\mathrm{sing}}$ and $V\not\in\Omega_g$. Conversely, if $U\leq V$ for some $U\in\Omega_g$ and $g^{-1}P_Jg=P_U<P_V$, then we must have $U<V$.  
\end{proof}

Now number the elements of $\mathcal{L}(J)$ ( i.e. write $\mathcal{L}(J)=\{g_1,g_2,g_3,\ldots  \}$ ) in accordance with the ordering $\preceq$ . Define $X^{P_J}=X^{P_J}_0$ and
 \[   X^{P_J}_n = X_{\mathrm{sing}}^{P_J}\cup \bigcup_{i>n} g_iK_{\Omega_{g_i}}.  \]
 for all $n>0$ and note that 
 $$X^{P_J}=X_{\mathrm{sing}}^{P_J}\cup \bigcup_{i>0} g_iK_{\Omega_{g_i}}.$$
 Moreover
 \begin{equation} \label{eq: limit3}  C^{\ast}_c(X^{P_J},X_{\mathrm{sing}}^{P_J}) =  \lim_{n}C^{\ast}(X^{P_J},X^{P_J}_n).   \end{equation}
 
 \begin{lemma} \label{lemma: important lemma3}For every $n\geq 1$, one has a natural isomorphism of chain complexes
 \[    \Psi_n: \mathrm{C}^{\ast}(X^{P_J}_{n-1},X^{P_J}_n) \xrightarrow{\cong} C^{\ast}(K_{\Omega_{g_n}},K_{>\Omega_{g_n}}) : f \mapsto \Psi(f), \]
 where $\Psi_n(f)(\sigma)=f(g_n\sigma)$ for every simplex $\sigma$  of  $K_{\Omega_{g_n}}$.
 \end{lemma}
 \begin{proof} 
 It is immediate from cellular excision that \[\mathrm{C}^{\ast}(X^{P_J}_{n-1},X^{P_J}_n) =C^{\ast}(g_nK_{\Omega_{g_n}}, g_nK_{\Omega_{g_n}}\cap X^{P_J}_n).\] We claim that $g_nK_{\Omega_{g_n}}\cap X^{P_J}_n= g_nK_{>\Omega_{g_n}}$. Let  $m>n$ and consider a simplex $g_n\sigma \in g_nK_{\Omega_{g_n}}\cap g_mK_{\Omega_{g_m}}$. If follows that for $V=S(\sigma)$,  $g_mP_{V}=g_nP_{V}$, $V\geq U_{g_n}\in\Omega_{g_n}$ and $V\geq U_{g_m}\in\Omega_{g_m}$.  Since $m>n$ and $g_n$ is the largest element of $g_nP_{U_{g_n}}$, we conclude that $V \in K_{>\Omega_{g_n}}$ thus $g_n\sigma\in g_nK_{>\Omega_{g_n}}$. By Lemma \ref{sets} we also have $g_nK_{\Omega_{g_n}} \cap X_{\mathrm{sing}}^{P_J}\subseteq {g_n}K_{>\Omega_{g_n}}$. All together, this implies that $g_nK_{\Omega_{g_n}}\cap X_n^{P_J}\subseteq g_n K_{>\Omega_{g_n}}$. Conversely, if $g_n \sigma \in g_n K_{>\Omega_{g_n}}$ then clearly $g_n\sigma \in g_nK_{\Omega_{g_n}}\cap  X_{n}^{P_J}$ since $g_n \sigma  \in X_{\mathrm{sing}}^{P_J}$. This proves the claim. Therefore there is a natural identification of  $C^{\ast}(X^{P_J}_{n-1},X^{P_J}_{n})$ with $C^{\ast}(g_nK_{\Omega_{g_n}},g_n K_{>\Omega_{g_n}})$. Combining this identification with the chain map induced by multiplication with $g_n$ we get the result.
 \end{proof}
 \begin{lemma}\label{eq: splitting3} 
For every $g \in \mathcal{L}(J)$ there is a  chain map
 \begin{equation} \rho^J_{g} : \mathrm{C}^{\ast}(K_{\Omega_g}, K_{>\Omega_g}) \rightarrow \mathrm{C}^{\ast}_{c}(X^{P_J},X_{\mathrm{sing}}^{P_J}) \end{equation}
 that it is given by
 \[\rho^{J}_{g}(f)(v\sigma)=  \Bigg\{
 \begin{aligned}
 & f(\sigma) \text{ if } \sigma \subseteq  K_{\Omega_{g}}\ \mbox{and} \ v \in P_{J}g\\
 &0\text{ otherwise.}\\
 \end{aligned}                 \] 
 \end{lemma}
\begin{proof} Note first that for any $U\in\Omega_g$, $P_U=g^{-1}P_Jg$ thus $gP_U=P_Jg$.
 Our first task is to show that $\rho^{J}_g(f)$ is well defined on $\mathrm{C}^{k}(K_{\Omega_g}, K_{>\Omega_g}) $, as a map of abelian groups. Take $v\sigma=v'\sigma \in X^{P_J}$. We need to check that $\rho^{J}_{g}(f)(v\sigma)$ equals $\rho^{J}_{g}(f)(v'\sigma)$. If $\sigma$ is not contained in $K_{\Omega_g}$, then both these expressions are zero. Now assume that $\sigma \subseteq K_{\Omega_g}$. If $\sigma$ is contained in $K_{>\Omega_g}$ then $\rho^{J}_{gB}(f)(v\sigma)$ and $\rho^{J}_{gB}(f)(v'\sigma)$ both equal zero, since $f(K_{>\Omega_g})=0$. Now assume that $\sigma$ is not contained in $K_{>\Omega_g}$. This means that $S(\sigma)\in\Omega_g$. Since $v\sigma =v' \sigma$ it follows that $vP_{S(\sigma)}=v'P_{S(\sigma)}$ and hence $vg^{-1}P_Jg=v'g^{-1}P_Jg$. This implies that $ vg^{-1} \in P_{J}$ if and only if  $v'g^{-1} \in P_{J}$, proving that $\rho^{J}_{g}(f)(v\sigma)$ equals $\rho^{J}_{g}(f)(v'\sigma)$. Since $P_{J}g$ and $K$ are finite,  it is also clear that $\rho^{J}_{g}(f)$ has compact support. Now let us prove that $\rho^{J}_{g}(f)(X_{\mathrm{sing}}^{P_J})=0$. Take $v\sigma \in X_{\mathrm{sing}}^{P_J}$. Then one has $\rho^{J}_{g}(f)(v\sigma)=0$, since in this case $\sigma$ is contained in $K_{\Omega_g}$ if and only if it is contained in $K_{>\Omega_g}$.

To prove that $\rho^{J}_{g}$ is a chain map, take $f \in \mathrm{C}^{n}(K_{\Omega_g}, K_{>\Omega_g}) $. We will show that $\rho^{J}_{g}(d(f))=d(\rho^{J}_{g}(f))$. Take $v\sigma \in X^{P_J}$. Then 
\begin{equation} \label{eq: the one guy}   d(\rho^{J}_{g}(f))(v\sigma)= \rho^{J}_{g}(f)(vd(\sigma))=\sum_{i=1}^k (-1)^{\mathrm{sgn}(\sigma_i)}\rho^{J}_{g}(f)(v\sigma_i), \end{equation}
where the $\sigma_i$ are the faces of $\sigma$. On the other hand
\begin{equation}\label{eq: the other guy}
 \rho^{J}_{g}(df)(v\sigma)=  \Bigg\{
 \begin{aligned}
 & \sum_{i=1}^k  (-1)^{\mathrm{sgn}(\sigma_i)} f(\sigma_i)\text{ if } \sigma \subseteq  K_{\Omega_g}\ \mbox{and} \ v \in P_{J}g\\
 &0\text{ otherwise.}\\
 \end{aligned}                 
\end{equation}
First, assume that $v\sigma \subseteq X^{P_J}_{\mathrm{sing}}$. In this case all the $v\sigma_i$ are also contained in $X^{P_J}_{\mathrm{sing}}$, from which it follows that (\ref{eq: the one guy})=(\ref{eq: the other guy})=0. Next, assume that $\sigma \subseteq K_{\Omega_g}$. In this case $\sigma_i \subseteq K_{\Omega_g}$ and hence (\ref{eq: the one guy})=(\ref{eq: the other guy}). The final remaining case is when $v\sigma$ is not contained in $X^{P_J}_{\mathrm{sing}}$ and $\sigma$ is not contained in $K_{\Omega_g}$. This implies that  $\sigma \subseteq K_{\Omega_r}$ with $P_{S(\sigma)}=r^{-1}P_Jr\neq g^{-1}P_Jg.$ In this case we must show that (\ref{eq: the one guy})=0. Let $W=S(\sigma)$ and assume that $\sigma$ has a face $\sigma_i$ that is contained in $K_{\Omega_g}$ but not in $K_{>\Omega_g}$. This implies that $U=S(\sigma_i)\in\Omega_g$. But then we have $W < U$ thus $r^{-1}P_Jr=P_W\subsetneq P_U=g^{-1}P_Jg$, which is a contradiction. We conclude that if $\sigma$ has a face $\sigma_i$ that is contained in $K_{\Omega_g}$, it must be contained in $K_{>\Omega_g}$ and therefore satisfies $f(\sigma_i)=0$. This shows that $(\ref{eq: the one guy})=0$. 
\end{proof}
 \begin{theorem} \label{th: main complex bredon}The map
 \[     \rho^{J}= \bigoplus_{g\in \mathcal{L}(J)} \rho^{J}_{g} :   \bigoplus_{g\in \mathcal{L}(J)}   \mathrm{C}^{\ast}(K_{\Omega_g},K_{>\Omega_g}) \rightarrow \mathrm{C}^{\ast}_{c}(X^{P_J},X_{\mathrm{sing}}^{P_J})   \]
 is an isomorphism of chain complexes. 
 \end{theorem}
 \begin{proof}
 Because  $g$ is the largest element in $P_{J}g$, it follows that for each $i\leq n$ the image of $\rho^J_{g_i}$ is contained in $C^{\ast}(X^{P_J},X^{P_J}_n)$. We now claim that for each $n \in \mathbb{N}$, there is an isomorphism of chain complexes
 \begin{equation}  \label{eq: this is it3}   \bigoplus_{i=1}^n \rho^J_{g_i} :   \bigoplus_{i=1}^{n}   \mathrm{C}^{\ast}(K_{\Omega_{g_i}}, K_{>\Omega_{g_i}}) \xrightarrow{\cong} C^{\ast}(X^{P_J},X^{P_J}_n).    \end{equation}
 We will prove this claim by induction on $n$. If $n$ is zero then the statement is obviously true. Now assume it is true for $n-1$ and consider the triple $(X^{P_J}, X^{P_J}_{n-1},X^{P_J}_{n})$. There is a short exact sequence of chain complexes
 \begin{equation*}    0 \rightarrow C^{\ast}(X^{P_J},X^{P_J}_{n-1}) \rightarrow C^{\ast}(X^{P_J},X^{P_J}_n) \rightarrow C^{\ast}(X^{P_J}_{n-1},X^{P_J}_n)\rightarrow 0.   \end{equation*}
 It now follows from Lemma \ref{lemma: important lemma3} and the induction hypothesis that there is a short exact sequence of chain complexes
 \begin{equation*}  0 \rightarrow \bigoplus_{i=1}^{n-1}   \mathrm{C}^{\ast}(K_{\Omega_{g_i}},K_{>\Omega_{g_i}}) \xrightarrow{\bigoplus_{i=1}^{n-1} \rho^{J}_{g_i}} C^{\ast}(X^{P_J},X^{P_J}_n) \xrightarrow{\Omega_n}  C^{\ast}(K_{\Omega_{g_n}},K_{>\Omega_{g_n}}) \rightarrow 0     \end{equation*}
 where $\Omega_n(f)(\sigma)=f(g_n\sigma)$ for every simplex $\sigma$  of $K_{\Omega_{g_n}}$.
 Since $\rho^J_{g_n}(f)(g_n\sigma)=f(\sigma)$ for every $f \in C^{\ast}(K_{\Omega_{g_n}},K_{>\Omega_{g_n}})$ and every simplex $\sigma$ of $K_{\Omega_{g_n}}$, it follows that $\rho^J_{g_n}$ is a splitting for this exact sequence, proving the claim. The theorem now follows from (\ref{eq: limit3}) and (\ref{eq: this is it3}).
 \end{proof}

  \section{Relative cohomological dimension}
Throughout this section, let $G$ be a discrete group that admits a contractible proper cocompact $G$-CW-complex $X$. This implies that $G$ is of type $FP_{\infty}$. If $G$ is virtually torsion-free then the virtual cohomological dimension of $G$ can be computed as (\ref{eq: vcd}). When $G$ is not virtually torsion-free the notion of virtual cohomological dimension is not defined, so one has to turn to a more general notion of cohomological dimension. In \cite{nucinkis99}, Nucinkis introduces a cohomology theory $\mathcal{F}\mathrm{H}^{\ast}(G,M)$ for groups $G$ and $G$-modules $M$, relative to the family of finite subgroups $\mathcal{F}$. The groups $\mathcal{F}\mathrm{H}^{\ast}(G,M)$ can be computed by taking the cohomology of a cochain complex obtained from applying $\mathrm{Hom}_G(-,M)$ to an $\mathcal{F}$-projective resolution $P_{\ast} \rightarrow \mathbb{Z}$. An $\mathcal{F}$-projective resolution $P_{\ast} \rightarrow \mathbb{Z}$ is by definition an exact chain complex of $G$-modules $P_{\ast} \rightarrow \mathbb{Z}$ that is split when restricted to every finite subgroup of $G$, and such that each $P_n$ is a direct summand of a $G$-module of the form $V\otimes \mathbb{Z}[\Delta]$, where $\Delta$ is the $G$-set $\coprod_{H \in \mathcal{F}}G/H$ and $V$ is a $G$-module. We remark that $\mathcal{F}\mathrm{H}^{\ast}(G,M)$ coincides with the Bredon cohomology group $\mathrm{H}^{\ast}_{\mathcal{F}}(G,\underline{M})$, where $\underline{M}$ is the fixed point functor associated to $M$ (e.g. see \cite[Prop. 3.1]{Degrijse}). 
\begin{definition} \rm \label{def: rel}
The cohomology theory $\mathcal{F}\mathrm{H}^{\ast}(G,-)$ yields the notion of relative cohomological dimension of $G$
\[  \mathcal{F}\mathrm{cd}(G)=  \sup\{ n \in \mathbb{N} \ | \ \exists M \in G\mbox{-Mod}   :  \mathcal{F}\mathrm{H}^n(G,M)\neq 0 \} .   \]
\end{definition}
It is immediate from the fixed point functor remark above that $\mathcal{F}\mathrm{cd}(G)\leq \underline{\mathrm{cd}}(G)$ and if $G$ is virtually torsion-free then it follows from the main theorem of \cite{MartinezNucinkis06} that $\mathcal{F}\mathrm{cd}(G)=\mathrm{vcd}(G)$. Moreover, by a result of Kropholler and Wall (see \cite[Th. 1.4]{KrophollerWall}) the augmented chain complex $C_{\ast}(X)\rightarrow \mathbb{Z}$ can be used to compute $\mathcal{F}$-relative cohomology of $G$. Since $X$ is finite dimensional, this implies that $\mathcal{F}\mathrm{cd}(G)< \infty$.  A theorem by St.~John-Green (\cite[Th. 3.11]{StJG}) says that if the relative cohomological dimension is finite, it coincides with the Gorenstein cohomological dimension, $\mathrm{Gcd}(G)$. By a result of Holm (see \cite[Th. 2.20]{Holm}), the Gorenstein cohomological dimension of a group can be computed as the largest $n$ for which the $n$-th cohomology of that group with coefficients in some projective module is not zero. Since every projective module is a direct summand of a free module and $G$ is of type $FP_{\infty}$, we conclude from all this that
\

\[  \mathcal{F}\mathrm{cd}(G)=  \sup\{ n \in \mathbb{N} \ | \   \mathrm{H}^n(G,\mathbb{Z}[G])\neq 0 \} .   \]
Since the cohomology with group ring coefficients of $G$ can be computed as the cohomology with compact support of $X$ (see \cite[Ch. VIII ex. 7.4]{Brown}), the following theorem is proven.

 \begin{theorem}\label{fcd} If $X$ is a contractible cocompact proper $G$-CW-complex, then 
\[ \mathcal{F}\mathrm{cd}(G) = \max \{  n \in \mathbb{N} \ | \ \mathrm{H}^n_{c}(X)\neq 0  \}. \]

 \end{theorem}

\section{Applications} 
The main application of our work in the previous sections is that we are able to compute the Bredon cohomological dimension of certain simple complexes of finite groups.
\begin{theorem}\label{th: main app} Let $G(\mathcal{Q})$ be a strictly developable complex of finite groups with development $X$. Denote $G=\widehat{G(\mathcal{Q})}$, $K=|\mathcal{Q}|$ and assume that $X^H$ is contractible for every finite subgroup of $G$. For $J\in\mathcal{Q}$, let
$$\Omega_J=\{U\in\mathcal{Q}\mid P_J=P_U\}.$$
Then
\[    \underline{\mathrm{cd}}(G)=\max\{ n \in \mathbb{N} \ | \  {\mathrm{H}}^{n}(K_{\Omega_J},K_{>\Omega_J})\neq 0 \ \mbox{for some} \ J \in \mathcal{Q}    \}.  \]

\end{theorem}
\begin{proof}
Since $X^H$ is contractible for every finite subgroup of $G$, it follows that $X$ is a cocompact model for proper actions of $G$. The point stabilizers of $X$ are exactly the conjugates of the local subgroups $P_J$, for $J \in \mathcal{Q}$. This implies that every finite subgroup of $G$ is contained in such a conjugate. Moreover, if $F$ is a finite subgroup of $G$ that is not conjugate to a local subgroup then it follows that $X^F=X^F_{\mathrm{sing}}$. Since fixed point sets of conjugate subgroups are homeomorphic via the obvious translation map,
 the theorem follows from Corollary \ref{cor: key cor} and Theorem \ref{th: main complex bredon}.
\end{proof}

In general, it is hard to determine when a given simple complex of finite groups is strictly developable, and even harder to decide whether or not its development has contractible fixed point sets. However, things become much more tractable in the presence of non-positive curvature. We refer the reader to \cite[Def. 12.26]{BridHaef} for the definition of a non-positively curved simple complex of groups.

\begin{theorem} \label{th: main app nonpos}Let $G(\mathcal{Q})$ be a non-positively curved complex of finite groups with $K=|\mathcal{Q}|$ simply connected and denote $G=\widehat{G(\mathcal{Q})}$. Then,
\[    \underline{\mathrm{cd}}(G)=\max\{ n \in \mathbb{N} \ | \ {\mathrm{H}}^{n}(K_{\Omega_J},K_{>\Omega_J})\neq 0 \ \mbox{for some} \ J \in \mathcal{Q}    \}.  \]

\end{theorem}
\begin{proof}  It follows from \cite[Th. 12.28]{BridHaef} that $G(\mathcal{Q})$ is strictly developable and its development is a complete CAT(0)-space on which $G$ acts by isometries. From Corollary II.2.8(1) in \cite{BridHaef} we conclude that $X^H$ is contractible for every finite subgroup $H$ of $G$. The statement is now immediate from Theorem \ref{th: main app}.

\end{proof}
This theorem can for example be used to determine when the fundamental group of a non-positively curved complex of finite groups is virtually free. This is illustrated in the following corollary for $n$-gons of groups. Briefly, an $n$-gon of groups is a simple complex of groups over the poset $\mathcal{Q}$, where $\mathcal{Q}$ is the poset of faces of the $2$-dimensional solid regular $n$-gon ordered by reverse inclusion  (e.g. see \cite[Ex. II.12.17(6)]{BridHaef}).
\begin{corollary} \label{ngon}The geometric dimension for proper actions of the fundamental group $G$ of a non-positively curved $n$-gon of finite groups is $2$. In particular, $G$ is not virtually free. 

\end{corollary}
\begin{proof} Since the developement of this $n$-gon of groups $G(\mathcal{Q})$ is a two-dimensional cocompact model for $\underline{E}G$, we have $\underline{\mathrm{gd}}(G)\leq 2$. Let $J$ be the element of the poset $\mathcal{Q}$ corresponding to the $2$-dimensional face of the $n$-gon. Then $\Omega_J=\{J\}$, thus $K_\Omega=K_J$ is contractible and $K_\Omega=K_{>J}$ is homeomorphic to a circle, since it corresponds to the outer edge of the $n$-gon. Therefore 
$\mathrm{H}^2(K_J,K_{>J})=\overline{\mathrm{H}}^1(K_{>J})\neq 0$. We conclude from Theorem \ref{th: main app nonpos} that $\underline{\mathrm{cd}}(G)\geq 2$ and hence $\underline{\mathrm{gd}}(G)= 2$. 

Since finitely generated virtually free groups $\Gamma$ act properly on a tree and hence satisfy $\underline{\mathrm{gd}}(\Gamma)\leq 1$, we deduce that $G$ is not virtually free. 

\end{proof}
Another application of Theorem \ref{th: main app nonpos} comes from considering groups acting properly and chamber transitively on buildings. Before we go into this, let us first recall some terminology and basic facts about buildings.\\

Let $(W,S)$ be a finite Coxeter system, i.e. $S=\{s_1,\ldots,s_n\}$ is a finite set of generators for the group $W$ with presentation
\[    W =  \langle s_1,\ldots,s_n \ | \ (s_is_j)^{m_{i,j}}\rangle \]
where $m_{i,j} \in \mathbb{N}\cup \{\infty\}$, $m_{i,i}=1$ and $m_{i,j}\geq 2$ if $i\neq j$. We refer to \cite{DavisBook} for a detailed introduction about Coxeter groups. Associated to a Coxeter group we have the poset  
$$\mathcal{Q}=\{J\subseteq S\mid\text{ the subgroup of $W$ generated by the elements of $J$ is finite}\},$$ 
ordered by inclusion. By convention, $\mathcal{Q}$ also contains the empty set as an initial object. The elements of $\mathcal{Q}$ are called spherical subsets. Now let $G$ be a discrete group together with subgroups $B, \{P_s\}_{s \in S}$ such that $B \subseteq P_s$ for all $s \in S$, and view the data $(G,B,\{P_s\}_{s\in S})$ as a chamber system (see \cite[Example 1.1]{Davis88}). For $J \subseteq S$ one defines the standard parabolic subgroup 
\[   P_J= \langle P_s \ | \ s \in J\rangle \subseteq G.  \] 

By definition we set $P_{\emptyset}=B$. If $J \in \mathcal{Q}$, then $P_J$ is called a spherical standard parabolic subgroup, while its conjugates are just called spherical parabolic subgroups. A coset $gP_J$ is called a $J$-residue. Now assume that $C=(G,B,\{P_s\}_{s\in S})$ is a building of type $(W,S)$ (see \cite[I.3]{Davis88}) such that all spherical parabolic subgroups of $G$ are finite, i.e. $P_J$ is a finite group for all $J \in \mathcal{Q}$. In this case $G$ acts properly and chamber transitively on the building $C$. Moreover any building of type $(W,S)$ that admits a chamber transitive automorphism group $G$ is of form $(G,B,\{P_s\}_{s\in S})$, where $B$ is the stabilizer of a chamber $c \in C$ and $P_s$ is the stabilizer of the $\{s\}$-residue containing $c$ (e.g. \cite[I.1]{Davis88} or see \cite[Ch. 5]{AB}). Note also that if $(W,S)$ is a finite coxeter system, then $W$ acts properly and chamber transitively on the building $(W,1,\{\langle s\rangle\}_{s \in S})$, which is of type $(W,S)$. \\

Consider the poset 
\[ \mathcal{P}= \{   gP_J \ | \ g \in G, J \in \mathcal{Q} \} \]
ordered by inclusion. We  can view $\mathcal{Q}$ as a subposet of $\mathcal{P}$ via $J\mapsto P_J$ and this map has a splitting given by $gP_J\mapsto J$. This implies that if we set $X=|\mathcal{P}|$, then $X$  is a simplicial complex with simplicial $G$-action that has
$K=|\mathcal{Q}|$ as a strict fundamental domain.  It is a theorem of Davis (\cite[Th. 11.1]{Davis88}) that $X$ supports a CAT(0)-metric such that $G$ acts on $X$ by isometries. It follows from the above and \cite[Cor. 12.22]{BridHaef} that $G$ is the fundamental group of the non-positively curved complex of groups over the poset $\mathcal{Q}$ with local groups $P_J$ where the maps $P_J\to P_{J'}$ whenever $J\subseteq J'$ are the obvious inclusions.  Observe also that  $\mathcal{P}$ can be identified with the poset (\ref{development}), since $gP_J \subseteq gP_U$ if and only if $J \subseteq U$. As the isotropy groups of the action of $G$ on $X$ are exactly the conjugates of the groups $P_J$, the fact that $X$ is a CAT-(0) space implies that $X$ is a cocompact model for $\underline{E}G$. Moreover, since $P_J\neq P_{J'}$ when $J\neq J'$, the set that we have denoted by $\Omega_J$ consists only of $J$. As $K_J$ is contractible we have $\mathrm{H}^n(K_J,K_{>J})=\overline{\mathrm{H}}^{n-1}(K_{>J})$. Stated in the world of buildings, Theorem \ref{th: main app nonpos} becomes the following result.

 \begin{theorem}\label{building} If $G$ acts properly and  chamber transitively on a building of type $(W,S)$, then
 \[\mathrm{vcd}(W)=\mathcal{F}\mathrm{cd}(G)=\underline{\mathrm{cd}}(G)= \max\{ n \in \mathbb{N} \ | \ \overline{\mathrm{H}}^{n-1}(K_{>J}) \neq 0 \ \mbox{for some} \ J \in \mathcal{Q}  \}.\]
In particular, the Bredon cohomological dimension and the virtual cohomological dimension of a Coxeter group coincide.
 \end{theorem}
 \begin{proof}  
It follows from Theorem \ref{th: main app nonpos} that
  \[   \underline{\mathrm{cd}}(G)= \max\{ n \in \mathbb{N} \ | \ \overline{\mathrm{H}}^{n-1}(K_{>J}) \neq 0 \ \mbox{for some} \ J \in \mathcal{Q}  \}.   \]
On the other hand, it follows from \cite[Cor. 8.2]{DDTJO}, \cite[Cor. 4.4]{Davis98}  and Theorem \ref{fcd} that $\mathrm{vcd}(W)= \mathcal{F}\mathrm{cd}(G)$.
Let $L$ be geometric realization of the simplicial complex with vertex set $S=\{s_1,\ldots,s_n\}$ and such that $\emptyset\neq J\subseteq S$ spans a simplex $\sigma_J$ if and only if $J\in\mathcal{Q}$. Then $K$ is the cone of the barycentric subdivision of $L$.
 Note that for each $J \in \mathcal{Q}$, $K_{>J}$ is homeomorphic to the link $\mathrm{Lk}(\sigma_J,L)$.  This link is by definition the geometric realization of the set of simplices $\tau$ of $L$ such that $\tau\cup\sigma_J$ is also a simplex.
 Since Dranishnikov proves (see \cite{Dranishnikov},\cite{Dranishnikov2}) that
 \[   \mathrm{vcd}(W)=  \max\{ n \in \mathbb{N} \ | \ \overline{\mathrm{H}}^{n-1}(\mathrm{Lk}(\sigma_J,L)) \neq 0 \ \mbox{for some} \ J \in \mathcal{Q}  \}, \]
 the theorem is proven.
 \end{proof}
Given a finite graph $\Gamma$ whose vertices $s \in S$ are labelled with finite groups $P_s$, the associated graph product $G$ is an example of a group acting properly and chamber transitively on a building of type $(W,S)$, where $(W,S)$ is the right angled coxeter groups determined by the graph $\Gamma$ (see \cite[Th. 5.1]{Davis88} and \cite[Ex. 12.30(2)]{BridHaef}). Since the kernel of the surjection $G \rightarrow \prod_{s \in S} P_s$ acts freely on the realization of the associated building, which is contractible, it follows that $G$ is virtually torsion-free.  The following corollary is therefore immediate from the previous theorem.

 \begin{corollary} Every finite graph product of finite groups $G$ satisfies $\mathrm{vcd}(G)=\underline{\mathrm{cd}}(G)$.
 
 \end{corollary}
We conclude this paper with a rephrasing of the question whether or not $\mathcal{F}\mathrm{cd}(G)=\underline{\mathrm{cd}}(G)$, when $G$ admits a cocompact model for proper actions.
\begin{question} \rm Does there exist a discrete group $G$ admitting a cocompact model for proper actions $X$ such that 
\[ \max \{  n \in \mathbb{N} \ | \ \mathrm{H}^n_{c}(X)\neq 0  \} <\max\{n \in \mathbb{N} \ | \ \exists K \in \mathcal{F} \ \mbox{s.t.} \ \mathrm{H}_c^{n}(X^K,X^K_\mathrm{sing})  \neq 0\}?\]

\end{question}

\end{document}